\documentclass[a4paper,11pt]{article}

\usepackage{amsmath,amssymb,amsfonts}
\usepackage{color}
\usepackage{hyperref}

\usepackage[utf8]{inputenc}

\newcommand{\R}{\mathbb R}
\newcommand{\C}{\mathbb{C}}
\newcommand{\Z}{\mathbb{Z}}
\newcommand{\N}{\mathbb{N}}
\newcommand{\Q}{\mathbb{Q}}
\newcommand{\A}{\mathbb{A}}
\newcommand{\K}{\mathbb{K}}

\newcommand{\conv}{\operatorname{conv}}

\newcommand{\supp}{\hbox{supp}}

\newcommand{\cA}{\mathcal{A}}
\newcommand{\cB}{\mathcal{B}}
\newcommand{\cD}{\mathcal{D}}
\newcommand{\cE}{\mathcal{E}}

\newcommand{\cP}{\mathcal{P}}
\newcommand{\cL}{\mathcal{L}}
\newcommand{\cZ}{\mathcal{Z}}
\newcommand{\cQ}{\mathcal{Q}}
\newcommand{\cR}{\mathcal{R}}

\newcommand{\cW}{\mathcal{W}}

\newcommand{\bff}{{\boldsymbol{f}}}
\newcommand{\bfx}{{\boldsymbol{x}}}
\newcommand{\bfp}{{\boldsymbol{p}}}
\newcommand{\bfq}{{\boldsymbol{q}}}
\newcommand{\bfv}{{\boldsymbol{v}}}
\newcommand{\bfw}{{\boldsymbol{w}}}

\newcommand{\bfS}{{\boldsymbol{S}}}
\newcommand{\bfF}{{\boldsymbol{F}}}

\newcommand{\Rem}{\hbox{Rem}}

\newcommand{\ord}{\operatorname{ord}}

\newtheorem{proposition}{Proposition}
\newtheorem{theorem}[proposition]{Theorem}

\newtheorem{remark}[proposition]{Remark}

\newtheorem{lemma}[proposition]{Lemma}

\newenvironment{proof}{
\trivlist \item[\hskip \labelsep\mbox{\it Proof.
}]}{\hfill\mbox{$\square$}
\endtrivlist}

\title{Sparse systems and algorithmic equidimensional decomposition\footnote{Partially supported by Universidad de Buenos Aires (UBACYT 20020190100116BA) and  CONICET (PIP 2021-2023 GI 11220200101015CO), Argentina.}}
\author{Mar\'ia Isabel Herrero{$^{1}$}, Gabriela Jeronimo{$^{2, 3, 4}$} and Juan Sabia{$^{3, 4}$}}
\date{}

\begin{document}

\maketitle
\begin{center}
\begin{minipage}{13cm}
\noindent {\small$^1$ Universidad Torcuato Di Tella. Departamento de Matem\'aticas y Estad\'istica. Buenos Aires, Argentina.}

\noindent {\small $^2$ Universidad de Buenos Aires. Facultad de Ciencias Exactas y Naturales. Departamento de Matem\'atica. Buenos Aires, Argentina.}

\noindent {\small $^3$  CONICET -- Universidad de Buenos Aires. Instituto de Investigaciones Matem\'aticas ``Luis A. Santal\'o'' (IMAS). Buenos Aires, Argentina.}

\noindent {\small $^4$ Universidad de Buenos Aires. Ciclo B\'asico Com\'un.  Departamento de Ciencias Exactas. Buenos Aires, Argentina.}
\end{minipage}
\end{center}

\medskip

\begin{flushright}
\emph{In memory of Joos Heintz, our mentor.}
\end{flushright}

\medskip

\begin{abstract}
    We present a new probabilistic algorithm that characterizes the equidimensional components of the affine algebraic variety defined by an arbitrary sparse polynomial system with prescribed supports. For each equidimensional component, the algorithm computes a witness set, namely a finite set obtained by intersecting the component with a generic linear variety of complementary dimension. The complexity of the algorithm is polynomial in combinatorial invariants associated to the supports of the polynomials involved.
\end{abstract}

\section{Introduction}

Due to the computational challenges involved in solving general polynomial equation systems,
which typically require substantial calculations, the study of polynomial systems with specific structure arising in applications has gained significant attention in the last decades.
In particular, systems of \emph{sparse} polynomials (that is, polynomials with nonzero coefficients only at  prescribed sets of monomials called their \emph{supports}) have become a central topic.

The foundational work of Bernstein \cite{Bernstein1975}, Kushnirenko \cite{Kushnirenko1976}, and Khovanskii \cite{Khovanskii1978} established that the number of isolated solutions in $(\C^*)^n$ of a polynomial system with $n$
equations in $n$ variables is bounded by a combinatorial invariant, known as the \emph{mixed volume}, which depends solely on the support sets of the polynomials involved.

This result has led to the development of algorithms for computing the isolated solutions of sparse polynomial systems, both numerically and symbolically (see, for instance, \cite{Verscheldeetal1994}, \cite{HS1995}, \cite{LiWang1996}, \cite{HS1997}, \cite{EV1999}, \cite{Rojas2003}, \cite{JMSW2009}, \cite{HJS2010}, \cite{DHJLLS2019}, \cite{BRSY2021}).

When solving arbitrary polynomial systems, one may be interested in characterizing not only the isolated solutions but also the components of higher dimensions of the affine variety defined by the system (for the best known complexity bounds for solving zero-dimensional systems see \cite{vdHL2021}).
Early work in this direction was presented in \cite{ChistovGrigoriev1983} and \cite{GiustiHeintz1991},
where the authors introduced new algorithms to compute the \emph{equidimensional decomposition} of an affine variety. Later progress on this topic can be found in \cite{ElkadiMourrain1999}, \cite{Lecerf2000}, \cite{JeronimoSabia2002}, \cite{Lecerf2003} and \cite{Jeronimoetal2004}, with shorter running time algorithms relying on probabilistic methods. More recently, a new equidimensional decomposition algorithm whose efficiency has been shown in practice was presented in \cite{Ederetal2023}.

The equidimensional decomposition problem has also been studied in the sparse setting in \cite{HJS2013}, both from the theoretical and the algorithmic points of view.  For \emph{generic} sparse systems with prescribed supports, the paper presents combinatorial conditions characterizing the equidimensional components of the algebraic variety defined by the system in the affine space, which are applied to design an efficient probabilistic algorithm to compute them. Also, some advances towards the algorithmic equidimensional decomposition for sparse systems with arbitrary coefficients were made in \cite{HJS2013}, but the problem was not completely solved, since the proposed algorithm provides partial information on the equidimensional components of the variety but it does not give a full description of them.

In this paper, we design a symbolic probabilistic algorithm that characterizes completely the equidimensional components of the affine variety defined by an \emph{arbitrary} sparse system.
This characterization is given by means of \emph{witness sets} of the equidimensional components of the variety,  which is a well-known approach in numerical algebraic geometry (see, for instance, \cite[Chapter 13]{SommeseWampler2005}). For an equidimensional variety $W\subset \A^n$ of dimension $k$, a witness set of $W$ is a finite set of points characterizing $W$ obtained by intersecting $W$ with a generic linear variety of co-dimension $k$.  Similarly to the notion of lifting fibers in the symbolic computation framework (see \cite{Lecerf2003}), witness sets enable one to perform geometric computations related to the variety (see also \cite{Sottile2020}).

Our algorithm computes, from a given sparse system $\bff = (f_1,\dots, f_m)$ of polynomials in $\Q[x_1,\dots, x_n]$, a family of finite sets $(\cP_k)_{0\le k \le n-1}$ such that, for every $k$, $\cP_k$ is a witness set of the equidimensional component of dimension $k$  of the variety defined by $\bff$ in $\A^n$. These finite sets are described by means of \emph{geometric resolutions} (see Section \ref{subsec:geomres} for a precise definition), a usual way of representing zero-dimensional varieties in the symbolic computation framework. In the particular case where all polynomials are supported in a set $\cA$, our main result, which follows straightforwardly from the more precise statement given in Theorem \ref{thm: main}, is:

 \bigskip
 \noindent \textbf{Theorem.} \emph{Let $f_1,\dots, f_m \in \Q[x_1,\dots, x_n]$ be polynomials supported on a finite set $\cA \subset (\Z_{\ge 0})^n$ that define an algebraic variety $V\subset\A^n$. Let $Q\subset \R^n$ be the convex hull of $\cA \cup \{0, e_1,\dots, e_n\}$, where $e_i$ is the $i$th vector of the canonical basis of $\R^n$. There is a probabilistic algorithm  that, for every $k=0, \dots, n-1$, computes a geometric resolution that represents a witness set of the equidimensional component of dimension $k$ of $V$.
The complexity of the algorithm is polynomial in $n$ and $\cD = n! \textrm{vol}_n(Q)$ and linear in $|\cA|$. }

\bigskip

Our approach is based on deformation techniques, which are widely used in numerical algebraic geometry for solving zero-dimensional polynomial systems (see, for instance, \cite{SommeseWampler2005}, \cite{Batesetal2013}). Homotopic deformation methods have also been applied in symbolic algorithms. In \cite{GiHe1993}, the authors introduced a new symbolic procedure for the computation of isolated points of algebraic varieties  that  significantly improved the previous complexity bounds. Combined with a symbolic Newton-Hensel lifting, these techniques lead to more efficient symbolic algorithms also solving some related elimination problems (see, for example, \cite{Giestieltal1998}, \cite{Heintzetal2000}, \cite{GLS01}, \cite{Schost2003}). Some algorithms for equidimensional decomposition of varieties also rely on these methods (see \cite{Lecerf2000}, \cite{JeronimoSabia2002}, \cite{Lecerf2003}, \cite{Jeronimoetal2004}). Algorithms designed specifically for the sparse polynomial setting are based on \emph{polyhedral deformations}, which take into account the Newton polytopes of the input polynomials (see \cite{Verscheldeetal1994}, \cite{HS1995}, \cite{HS1997}, \cite{JMSW2009}, \cite{HJS2010}, \cite{HJS2013}).

In this paper, we apply homotopic deformations together with polyhedral deformation-based subroutines (mainly from \cite{JMSW2009} and \cite{HJS2013}) to obtain a probabilistic symbolic algorithm with a complexity bound depending on combinatorial parameters (mixed volumes) for characterizing the equidimensional components of affine varieties defined by arbitrary sparse polynomial systems. To this end, we first prove some theoretical results on varieties defined by particular homotopic deformations which we then use in our algorithm. Relying on these results, we design a new probabilistic subroutine that given a sparse system of polynomials defining a variety $V\subset \A^n$ and a finite set $\cP\subset V$, computes for $k=0,\dots, n-1$, the set of all points in $\cP$ that lie in the equidimensional component of $V$ of dimension $k$ and in no component of $V$ of larger dimension.
This new subroutine is the key ingredient that enables us to obtain witness sets for the equidimensional components of the variety from witness supersets for them (c.f. \cite{HJS2013}). Our complexity bounds are \emph{polynomial} in terms of \emph{mixed volumes}, which allows one to obtain better running-time estimates when dealing with structured polynomial systems.

The paper is organized as follows: in Section \ref{sec:preliminaries}, we recall the basic notions on sparse polynomial systems, equidimensional decomposition of affine varieties and state the complexities of several basic algorithmic computations with polynomials and matrices we will use. Section \ref{sec:theoretical} is devoted to proving the main theoretical results of the paper.  In Section \ref{sec:algorithms}, we first discuss the notion of a geometric resolution from an algorithmic point of view, and finally, we present our algorithms and we prove their correctness and complexity bounds.

\section{Preliminaries}\label{sec:preliminaries}

\subsection{Sparse polynomial systems} \label{subsec:sparse}

Let $x=(x_1,\dots, x_n)$ be a family of indeterminates over a field $\K$ and $\K[ x] = \K[x_1,\dots, x_n]$ be the ring of polynomials in the variables $x$ with coefficients in $\K$. Throughout this paper, we will work with fields with characteristic $0$, mainly with $\K= \Q$ or $\K = \Q(t)$ (the field of fractions of the polynomial ring in one variable with coefficients in $\Q$).

For $a\in (\Z_{ \ge 0})^n$, we write $x^a = x_1^{a_1}\dots x_n^{a_n}$ for the corresponding monomial in the variables $x_1, \dots, x_n$.

Given a family $\cA = (\cA_1, \dots, \cA_m)$ of finite subsets of $(\Z_{\ge 0})^n$, a \emph{sparse polynomial system supported on $\cA$} with coefficients in $\K$ is a family $\bff= (f_1,\dots, f_m)$ of polynomials in $\K[ x_1,\dots, x_n]$ of the form
\[ f_j(x) = \sum_{ a\in \cA_j} c_{ja} x^a, \ c_{ja} \in \K \mbox{ for all } a\in\cA_j, \ j=1,\dots, m; \]
that is, the only monomials that can appear in $f_j$ with a nonzero coefficient are those with exponents given by the points in $\cA_j$.

Assume $m=n$. We denote by $MV_n(\cA_1, \dots, \cA_n)$, or simply $MV_n(\cA)$, the \emph{mixed volume} of the family of lattice polytopes $\conv(\cA_1), \dots, \conv(\cA_n)$, the convex hulls in $\R^n$ of $\cA_1, \dots, \cA_n$. For a definition of mixed volume, we refer the reader to \cite[Chapter 7]{CLO2}. The mixed volume of $n$ polytopes in $\R^n$ is a symmetric function, linear in each coordinate with respect to the Minkowski sum, and monotone with respect to inclusion, that is,
$MV_n(\cA_1, \dots, \cA_n) \le MV_n(\cB_1,\dots, \cB_n)$ if $\cA_j\subseteq \cB_j$ for every $1\le j \le n$. If $\conv(\cA_1)=\dots = \conv(\cA_n) = C$, the equality $MV_n(\cA_1, \dots, \cA_n) = n! \hbox{vol}_n(C)$ holds.

The BKK theorem (\cite{Bernstein1975}, \cite{Khovanskii1978}, \cite{Kushnirenko1976}) states that the mixed volume $MV_n(\cA_1,\dots, \cA_n)$ is equal to the number of common zeros in $(\C^*)^n$ of a generic polynomial system supported on $\cA_1,\dots, \cA_n$ with coefficients in $\C$; furthermore, it is an upper bound for the number of isolated solutions in $(\C^*)^n$ of a polynomial system supported on $\cA_1,\dots, \cA_n$ with \emph{arbitrary} complex coefficients.

\subsection{Affine varieties, equidimensional decomposition and witness sets}

A family of polynomials $\bff= (f_1, \dots, f_m) $ in $\K[x_1, \dots, x_n]$ defines an algebraic variety in the affine space $\A^n = \A^n(\overline{\K})$ over an algebraic closure $\overline{\K}$:
\[V= V(\bff) = \{ \bfx \in \A^n \mid f_j(\bfx) = 0 \hbox{ for all } 1\le j\le m\},\]
which can be decomposed uniquely as a finite non-redundant union of irreducible varieties in $ \A^n $. For $k=0,\dots, n-1$, let $V_k\subset \A^n$ be the union of all the irreducible components of $V$ of dimension $k$. Then, we obtain the \emph{equidimensional decomposition} of $V$:
\[V= \bigcup_{0\le k \le n-1} V_k.\]
For every $k$, the set $V_k$ is an algebraic variety called the \emph{equidimensional component of dimension $k$} of $V$. Note that, for some values of $k$, this component may be empty.

The intersection of an equidimensional variety $W$ of dimension $k$ with a generic linear affine  variety $L$ of codimension $k$ is a finite set
$$\mathcal{P}= W \cap L,$$
called a \emph{witness set} of $W$. All witness sets of an equidimensional variety $W$ have the same cardinality, which is the \emph{degree} of $W$, denoted by $\deg(W)$. Moreover, $\deg(W)$ is the maximum number of points in the intersection of $W$ with a linear affine variety of codimension $k$, provided that the intersection is finite (see \cite{Heintz83}). The notion of witness set was introduced in the numerical algebraic geometry framework as a way to encode positive dimensional components of algebraic varieties. Indeed, if $W$ is an equidimensional component of the algebraic variety defined by a polynomial system $\bff$, the data $(\cP, L, \bff)$ characterizes $W$ (see \cite[Chapter 13]{SommeseWampler2005}).

The degree of an arbitrary variety $V\subset \A^n$ is defined as the sum of the degrees of its equidimensional components (see \cite{Heintz83}). If $V$ is defined by a sparse polynomial system $\bff= (f_1, \dots, f_n)$ supported on $\cA=(\cA_1, \dots, \cA_n)$, according to \cite[Theorem 16]{HJS2013}, we have
\[\deg(V) \le MV_n(\cA_1\cup \Delta_n, \dots, \cA_n \cup \Delta_n),\]
where $\Delta_n = \{ 0, e_1, \dots, e_n\}$ is the vertex set of the standard unitary simplex of $\R^n$ (here, $e_i$ is the $i$th vector of the canonical basis of $\R^n$).

\subsection{Complexity estimates}

Our algorithms work mainly with univariate and multivariate polynomials with rational coefficients. The notion of \emph{complexity} we consider is the number of arithmetic operations and comparisons in $\Q$ that the algorithm performs.

We encode a polynomial $f\in \Q[x_1,\dots, x_n]$ in one of the following ways:
\begin{itemize}
    \item \emph{dense representation}: given an upper bound $d\in \Z_{ \ge 0}$ for $\deg(f)$, by means of the vector of all the coefficients of $f$ in a pre-fixed order of the monomials of degree at most $d$;
    \item \emph{sparse representation}: given a finite set $\cA\subset (\Z_{ \ge 0})^n$ such that $f$ is supported on $\cA$, by means of the vector $((c_a, a); a\in \cA)$, where $c_a$ is the coefficient of the monomial $x^a$ in $f$;
    \item \emph{straight-line program representation} (slp for short):  by means of an algorithm without branchings that enables the evaluation of $f$ at any given point (for a precise definition, see \cite[Section 4.1]{BCS1997}).
\end{itemize}
Starting from the foundational work of Giusti and Heintz (see \cite{GiHe1993}), straight-line programs proved to be effective in the construction of algorithms to solve many algebraic and geometric problems. This is the reason why we use this way of encoding in our intermediate computations.

In our complexity estimates, we use the standard $O$-notation:
for $f, g : \Z_{\ge 0} \to \R$, $f (d) = O(g(d))$ if $|f (d)| \le c|g(d)|$ for a positive constant $c$.

We point out that any polynomial
$f \in \Q[x_1,\dots, x_n]$ of degree at most $d >0$ supported on a set $\cA\subset (\Z_{ \ge 0})^n$
can be evaluated by means of an slp of length $O(n|\cA| \log(d))$ that can be obtained straightforwardly from the sparse representation of $f$.
Here and in the sequel, we write $\log$ for the logarithm to the base $2$.

For the reader's convenience, we summarize the complexity estimates of some basic computations with univariate polynomials with coefficients in a field $k$ we will use. Following \cite{JMSW2009}, we use the notation
$M(d) := d \log^2(d) \log\log(d)$.

Multiplication and division with remainder of polynomials in $k[Y]$ of degrees bounded by $d$ can be computed with $O(M(d)/\log(d))$ arithmetic operations (see \cite[Corollary 7.19 and Theorem 9.6]{vzGG2013}).
The fast Extended Euclidean Algorithm (see \cite[Section 11]{vzGG2013})  computes the gcd of two polynomials $f, g \in k[Y]$ of degrees bounded by $d$ within complexity $O(M(d))$. We compute the gcd of polynomials $h_0, \dots, h_m\in k[Y]$ of degrees bounded by $d$ either in the standard recursive way, which takes $O(m M(d))$ operations, or in a probabilistic way, as $\gcd(h_0, \sum_{1\le j\le m} c_j h_j)$ for randomly chosen constants $c_1, \dots, c_m$ (see \cite[Algorithm 6.45]{vzGG2013}), within complexity $O(md+M(d))$ (see Section \ref{sec:algorithms} for the notion of probabilistic algorithm we use).

Multipoint evaluation and interpolation of a polynomial $f\in k[Y]$ of degree $d$ at $d+1$ distinct points  can be achieved within complexity $O(M(d))$ (see \cite[Corollaries 10.8 and 10.12]{vzGG2013}).
For a rational function $f= p/q$, where $p, q\in k[Y]$ have degrees bounded by $d$,  we compute the dense representation of $p$ and $q$ from its Taylor series expansion up to order $2d$ using Pad\'e approximation with $ O(M(m))$ arithmetic operations in $k$ (see \cite[Section 5.9]{vzGG2013}).

Any subresultant polynomial of two polynomials $f, g \in k[Y]$ of degrees bounded by $d$ (in particular, their resultant) can be computed with $O(M(d))$ operations in $k$ (see \cite[Corollary 11.19]{vzGG2013}).

In addition to manipulating polynomials, we will need to perform matrix computations. We denote $\Omega$ the exponent of the complexity of matrix multiplication, that is, a real number such that the product of two $n\times n$ matrices can be computed within complexity $O(n^\Omega)$ (see \cite[Section 12.1]{vzGG2013}). For practical computations, Strassen's multiplication algorithm with $\Omega = \log 7$ can be used. The complexity of the computation of the  determinant, the inverse and the characteristic polynomial of an $n\times n$ matrix can also  be computed within $O(n^{\Omega})$ arithmetic operations.

\section{Theoretical results} \label{sec:theoretical}

Let $\bff=(f_1, \dots, f_m)$ be a finite family of polynomials in $\Q[x_1, \dots, x_{n}]$ not all zero. Consider the variety $V = V(\bff) \subset \C^{n}$ and its equidimensional decomposition $$V = \bigcup_{k=0}^{n-1} V_k.$$

To characterize the equidimensional components $V_k$, we will first obtain witness supersets of each of them (that is, a finite set of points containing a witness set) and then remove all additional points from these witness supersets.
To this end, for a given a finite set $\cP\subset V$, we want to determine, for $k=0,\dots, n-1$, the set of all points in $\cP$ that lie in the equidimensional component of $V$ of dimension $k$ and in no component of $V$ of larger dimension. In this section, we will focus on this problem from a theoretical point of view.

\subsection{A single point}\label{sec:onepoint}

Given a point $\bfp \in V$, our aim is to determine the equidimensional component of the largest dimension that contains $\bfp$.

Fix $k$, with $1\le k \le n-1$. Take $\ell_1, \dots, \ell_k$ generic linear polynomials in $\Q[x_1,\dots, x_n]$ and set
\begin{align*}
L_0&= \{ \bfx \in \A^n \mid \ell_i(\bfx) -\ell_i(\bfp)=0 \ \hbox{for all } 1\le i \le k\} \\
L_1&=\{ \bfx\in \A^n \mid \ell_i(x) = 0\ \hbox{for all } 1\le i \le k\}
\end{align*}
Note that $L_0$ is a generic linear variety of codimension $k$ containing the given point $\bfp$.

The generic conditions we will assume on $\ell_1, \dots, \ell_k$ are the following:
\begin{description}
\item[{\small {\rm (H1)}}] $V_j \cap L_1$
is either empty or $(j-k)$-equidimensional for all $j \ge k$, and $V_j \cap L_1 = \emptyset$ for all $j < k$.
\item[{\small {\rm (H2)}}] $V_{k} \cap (\bigcup_{j>k} V_j) \cap L_1 = \emptyset$,
\item[{\small {\rm (H3)}}] $V_j \cap L_0$ is either empty or  $(j-k)$-equidimensional for all $j \ge k$,  and $V_j \cap L_0$ is either empty or equal to $\{\bfp\}$ for all $j < k$.
\end{description}

Consider a new parameter $s$ and the varieties
\begin{align*}
\widehat{V}&= V\times \A^1 = \{(\bfx,s) \in\A^{n+1} \mid \bfx \in V\}\\
    \cL &= \{ (\bfx,s)\in \A^{n+1} \mid \ell_i(\bfx) - (1-s) \ell_i(\bfp) = 0 \hbox{ for all } 1\le i \le k\}
\end{align*}

The variety $\cL$ can be interpreted as a family of linear varieties parameterized by $s$, with
$\cL \cap \{s=s_0\} = L_{s_0} \times \{s_0\}$ for every $s_0\in \A^1$, where
\[L_{s_0}= \{ \bfx \in \A^n \mid \ell_i(\bfx) -(1-s_0)\ell_i(\bfp)=0 \ \hbox{for all } 1\le i \le k\}. \]
In particular, $\cL \cap \{s=0\} = L_0\times \{0\}$ and $\cL \cap \{s=1\} = L_1 \times \{1\}$.

Note that, if the coefficients of $\ell_1, \dots, \ell_k$ are chosen in a Zariski open set ensuring that (H1) and (H2) hold for $L_1$, then $L_{s_0}$ satisfies the same conditions for generic $s_0\in \A^1$.

\smallskip
We have that $\widehat{V}= \bigcup_{j=0}^{n-1} \widehat{V}_j$, where $\widehat{V}_j = V_j \times \A^1$, that is, $\widehat{V}_j$ is the equidimensional component of dimension $j+1$ of $\widehat{V}$.

Let $\cZ\subset \A^{n+1}$ be the union of the irreducible components $C$ of $\widehat{V}\cap \cL$ such that $\overline{\pi_s(C)}=\A^1$, where $\pi_s\colon \A^{n+1} \to \A^1$, $\pi_s(\bfx, s) = s$. We write $\cZ= \bigcup_{h=0}^{n} \cZ_h$ for the equidimensional decomposition of $\cZ$.

\begin{lemma}\label{lem:equidecZ}
With the previous notation, we have that $\cZ_h \subset V_{h+k-1}\times \A^1$ for $h=1, \dots, n-k$, and $\cZ_h= \emptyset$ otherwise.
\end{lemma}

\begin{proof}
Let $V= \bigcup_W W$ be the irreducible decomposition of $V$. Then,
$\widehat{V} \cap \cL = \bigcup_{W}  (\widehat{W} \cap \cL)$, where $\widehat{W} = W \times \A^1$ for each $W$. In particular, each irreducible component $C$ of $\cZ$ is an irreducible component of $\widehat{W} \cap \cL$ for some irreducible component $W$ of $V$.

Let $W\subset \A^n$ be an irreducible component of $V$ and consider an irreducible component $C$ of $\widehat{W} \cap \cL$ such that $\overline{\pi_s(C)} = \A^1$. We claim that $\dim(C) = \dim(W)+1-k$.

First note that, since $\cL$ is defined by $k$ equations, all the irreducible components of $\widehat{W} \cap \mathcal{L}$ have dimension at least $\dim(\widehat{W}) -k = \dim(W)+1-k$.
If $\overline{\pi_s(C)} = \A^1$, for generic $s_0$, we have that $C\cap \{s=s_0\}$  is a non-empty equidimensional variety of dimension $\dim(C)-1$. Now,
$C\cap \{s =s_0\}\subset \widehat{W}\cap \cL \cap \{ s = s_0\} = (W\cap L_{s_0}) \times \{s_0\}$.
By the genericity of $L_{s_0}$ it follows that $W\cap L_{s_0}$ is either empty, if $\dim(W) <k$, or an equidimensional variety of dimension $\dim(W) - k$, if $\dim(W) \ge k$. We conclude that $\dim(W) \ge k$ and $\dim(C) -1= \dim(C\cap \{s=s_0\})  \le \dim(W) - k$. This finishes the proof of the claim.

We conclude that every component $C$ in $\cZ_h$ is an irreducible component of $\widehat{W} \cap \cL$, for an irreducible component $W$ of $V$ such that $\dim(W) = h+k-1$ and so, $\cZ_h \subset \widehat{V}_{h+k-1}\cap \cL \subset {V}_{h+k-1}\times \A^1$. In particular, if $\cZ_h\ne \emptyset$, then $k \le h+k-1 \le \dim(V) = n-1$, that is, $1\le h\le n-k$.
\end{proof}

\begin{lemma}\label{lem:incurves}
Let $\bfq\in V_k \cap L_0$ be a point such that $\bfq \notin \bigcup_{j>k} V_j$. Then $(\bfq, 0) \in \cZ_1$.
\end{lemma}

\begin{proof}
As $(\bfq, 0) \in (V_k \cap L_0) \times \{ 0\} \subset \widehat{V}_k \cap \cL \subset \widehat{V} \cap \cL$ and $\widehat{V}_k \cap \cL$ does not have isolated points, there is an irreducible component $C$ of $\widehat{V} \cap \cL$ with $(\bfq, 0) \in C$ and $\dim(C)\ge 1 $.

Now,  $(\bfq,0) $ is an isolated point of
\[\widehat{V} \cap \cL \cap \{s=0\} = \Big(\bigcup_{j \le k} (V_j \cap L_0) \cup \bigcup_{j>k} (V_j\cap L_0)\Big) \times \{0\},\] because of the assumption $\bfq \notin \bigcup_{j>k} V_j $ and the genericity of $L_0$, which implies that $\bigcup_{j \le k} (V_j \cap L_0) $ is a finite set.
Then, $C \not\subset \{ s=0\}$ and $\dim(C) =1$. Therefore, $C\subset \cZ_1$.
\end{proof}

\bigskip
The previous lemmas enable us to detect the equidimensional component of $V$ of the largest dimension containing a point $\bfp\in V$. Consider $\ell_1,\dots, \ell_{n-1}$ generic linear polynomials in $\Q[x_1,\dots, x_n]$ so that assumptions (H1), (H2) and (H3) are met for every $1\le k \le n-1$.

For each $k$, with $1\le k \le n-1$, consider the variety $\cZ^{(k)}\subset \A^{n+1}$ defined from $V$, $\ell_1, \dots, \ell_k$ and $\bfp$ as above, and its equidimensional decomposition  $\cZ^{(k)} = \bigcup_{h=1}^{n-k} \cZ^{(k)}_h$.

\begin{proposition}\label{prop:dimpunto}
With the previous assumptions and notation, for a point $\bfp \in V$, we have:
\begin{itemize}
    \item $\bfp\in V_{n-1}$ if and only if $(\bfp,0) \in \cZ^{(n-1)}$.
    \item If $\bfp\notin \bigcup_{j>k} V_j$, the following are equivalent:
    \begin{itemize}
        \item $\bfp\in V_k$,
        \item $(\bfp,0)\in \cZ^{(k)}_1$,
        \item $(\bfp,0)\in \cZ^{(k)}$.
    \end{itemize}
\end{itemize}
\end{proposition}

\begin{proof}
By Lemma \ref{lem:equidecZ} for $k=n-1$, we have that $\cZ^{(n-1)}=\cZ_1^{(n-1)} \subset V_{n-1} \times \A^1$.

Assume now that $\bfp\notin \bigcup_{j>k} V_j$ for a fixed $k$ (this holds, in particular, for $k=n-1$). Since $\bfp\in L_0$, Lemma \ref{lem:incurves} ensures that if $p\in V_{k}$,  then $(\bfp,0)\in \cZ_1^{(k)} \subset \cZ^{(k)}$. On the other hand, due to Lemma \ref{lem:equidecZ}, $\cZ^{(k)} \subset \bigcup_{j\ge k} V_j \times \A^1$. Then, if $(\bfp, 0) \in \cZ^{(k)}$, it follows that $\bfp \in V_k$.
\end{proof}

\subsection{A finite set of points}\label{subsec:finiteset}

In our algorithms we will not be able to work with individual points, but we will have to deal with finite sets of points represented by geometric resolutions. To do this, we will generalize our previous results to this more general setting.

Assume we have a finite set of points $\cP \subset V$. Our aim is, as before, to introduce a family of algebraic varieties that enables us to determine, for every $\bfp\in \cP$, which is the largest equidimensional component of $V$ that contains $\bfp$.

For every $\bfp \in \cP$, let $\theta_{\bfp}\in \C$ such that $\theta_{\bfp}\ne \theta_{\bfp'}$ for $\bfp \ne \bfp'$.
As in Section \ref{sec:onepoint}, fix $k$ with $1\le k \le n-1$,  take $\ell_1, \dots, \ell_{k}$ (generic) linear affine functions satisfying (H1), (H2) and (H3),
and consider
\begin{itemize}
\item $\widehat V_{\bfp} = \{ \theta_{\bfp}\} \times V \times \A^1 \subset \A^{n+2}$
\item $\cL_\bfp = \{\theta_\bfp\} \times \{ (\bfx,s)\in \A^{n+1} \mid \ell_i(\bfx) - (1-s) \ell_i(\bfp) = 0 \hbox{ for all } 1\le i \le k\} \subset \A^{n+2}$
\item $\cZ_\bfp\subset \A^{n+2}$ the union of the irreducible components $C$ of $\widehat V_{\bfp} \cap \cL_\bfp$ such that $\overline{\pi_s(C)} = \A^{1}$.
\end{itemize}

We may also assume that
\begin{description}
\item[{\small {\rm (H4)}}] $\bfp' \notin  \{ \bfx\in \A^{n} \mid \ell_i(\bfx) -\ell_i(\bfp) = 0 \hbox{ for all } 1\le i \le k\}$ for $\bfp' \in \cP$, $\bfp' \ne \bfp$.
\end{description}

Let $\Theta= \{ \theta_{\bfp}\mid \bfp \in \cP\}$ and $V_\cP = \Theta \times V \subset \A^{n+1}$.
Consider the following algebraic varieties in $\A^{n+2}$:
\[
\widehat{V}_\cP= V_\cP \times \A^1 = \bigcup_{\bfp \in \cP} \widehat V_{\bfp}, \qquad
    \cL_\cP = \bigcup_{\bfp \in \cP} \cL_{\bfp}, \qquad \cZ_\cP = \bigcup_{\bfp \in \cP} \cZ_{\bfp}.\]
Since $\widehat{V}_\cP \cap \cL_\cP = \bigcup_{\bfp\in \cP} V_{\bfp} \cap \cL_{\bfp}$, we have that $\cZ_\cP$ is the union of the irreducible components $\mathcal{C}$ of $\widehat{V}_\cP \cap \cL_\cP$ such that $\overline{\pi_s(\mathcal{C})} = \A^1$.

Note that the equidimensional decomposition of $V_\cP$ is $V_\cP = \bigcup_{j=0}^{n-1} (\Theta \times V_j)$, that is, the equidimensional component of $V_{\cP}$ of dimension $j$ is $V_{\cP, j} =\Theta \times V_j$, for $j=0,\dots, n-1$.
Let $\cZ_\cP = \bigcup_{0\le h \le n+1} \cZ_{\cP, h}$ be the equidimensional decomposition of $\cZ_\cP$. From Lemmas \ref{lem:equidecZ} and \ref{lem:incurves}, we deduce:

\begin{lemma}\label{lem:severalpoints}
With the previous notation, we have:
\begin{enumerate}
    \item[(a)] $\cZ_{\cP,h} \subset V_{\cP, h+k-1} \times \A^1$ for $h=1,\dots, n-k$ and $\cZ_{\cP,h} = \emptyset$ otherwise.
    \item[(b)] For every $\bfp\in \cP$, if $\bfp\in V_k$ and $\bfp \notin \bigcup_{j>k}V_j$, then $(\theta_\bfp, \bfp, 0)\in \cZ_{\cP,1}$. \hfill $\square$
\end{enumerate}
\end{lemma}

We will now consider the previous construction for every $1\le k \le n-1$. Let $\ell_1, \dots, \ell_{n-1}$ be generic linear polynomials in $\Q[x_1, \dots, x_n]$ so that assumptions (H1), (H2), (H3) and (H4) are met for every $1\le k\le n-1$. For each $k$, let $\cZ_\cP^{(k)}\subset \A^{n+2}$ be a variety defined from the variety $V$, the finite set $\cP$, and $\ell_1, \dots, \ell_k$ as before.

From the previous lemma we deduce the following result that extends Proposition \ref{prop:dimpunto}:
\begin{proposition}\label{prop:defseveralpoints}
Let $\Pi: \A^{n+2} \to \A^{n+1}$ be the projection $(\theta, x, s) \mapsto (x,s)$.
With the previous assumptions and notation, for every $p\in \cP$ we have:
\begin{itemize}
    \item $p\in V_{n-1}$ if and only if $(p,0) \in \Pi(\cZ_\cP^{(n-1)})$.
    \item If $p\notin \bigcup_{j>k}V_j$, then $p\in V_k$ if and only if $(p,0) \in \Pi(\cZ_\cP^{(k)})$.
\end{itemize}
\end{proposition}

\begin{proof}
The result follows from Lemma \ref{lem:severalpoints} taking into account that our genericity assumption (H4) implies that, for $\bfp \in \cP$ and $\theta \in \Theta$,
$(\theta, \bfp, 0) \in V_\cP \cap \cL_\cP$ if and only if $\theta = \theta_p$: indeed, if $(\theta, \bfp, 0)\in V_\cP \cap \cL_\cP$, there exists $\bfp' \in \cP$ with $\theta = \theta_{\bfp'}$ and $(\theta, \bfp, 0) \in V_{\bfp'} \cap \cL_{\bfp'}$; in particular, $\bfp \in \cL_{\bfp'} \cap \{ s=0\} = \{ \bfx \in \A^n \mid\ell_i(\bfx) - \ell_i(\bfp') =0 \hbox{ for all } 1\le i \le k\}$ and so, $\bfp' =\bfp$.
\end{proof}

\section{Algorithms}\label{sec:algorithms}

The algorithms we introduce in this paper are probabilistic, in the sense that they make random choices of points which lead
to a correct computation provided the points lie outside certain proper Zariski closed sets of suitable
affine spaces. Their error probability could be controlled by making these random choices within sufficiently large
sets of integer numbers whose size depend on the degrees of the polynomials defining the previously
mentioned Zariski closed sets, by means of the Schwartz–Zippel lemma (\cite{Schwartz1980}, \cite{Zippel1993}).

This algorithmic model has been widely used for solving elimination problems dealing with polynomial systems and, in particular, to obtain the
first complexity bounds polynomial in the output size for equidimensional decomposition (see \cite{JeronimoSabia2002}, \cite{Lecerf2003}).

\subsection{Geometric resolutions} \label{subsec:geomres}

Our algorithms deal with finite sets of points which are the zero sets of systems of multivariate polynomials. Thus, we need to represent them in a way suitable for algorithmic purposes. A representation that goes back to Kronecker (\cite{Kronecker1882}) and is widely used in computer algebra is a \emph{geometric resolution} (see, for instance, \cite{GiHe1993}, \cite{GLS01} and the references therein).

Let $\cP\subset \A^n$ be a zero-dimensional algebraic variety definable over $\K$. A \emph{geometric resolution of $\cP$} consists in a polynomial parametrization of the points of $\cP$ by the roots of a univariate polynomial $Q\in \K[Y]$.

Given a linear form $\ell\in \K[x_1, \dots, x_n]$ that separates the points of $\cP$, that is, such that $\ell(\bfp) \ne \ell(\bfp')$ for $\bfp\ne \bfp' $ in $\cP$ we consider its \emph{minimal polynomial} $Q(Y) = \prod_{\bfp \in \cP} (Y-\ell(\bfp))$.
We will consider two different kinds of geometric resolution, differing in the way the points of $\cP$ are parametrized from the zeros of $Q$:
\begin{itemize}
    \item \emph{Shape-lemma representation:} $(Q(Y), v_1(Y), \dots, v_n(Y))$ with $\deg(v_i) < |\cP|$ for $i=1,\dots, n$, such that
    \[\cP = \{ (v_1(y),\dots, v_n(y))\in \A^n \mid y \in \overline{\K}: Q(y) = 0 \}\]
    \item \emph{Kronecker representation:} $(Q(Y), w_1(Y), \dots, w_n(Y))$ with $\deg(w_i) < |\cP|$ for $i=1,\dots, n$, such that
    \[\cP = \Big\{ (\frac{\partial Q}{\partial Y} (y))^{-1} (w_1(y),\dots, w_n(y))\in \A^n \mid y \in \overline{\K}: Q(y) = 0 \Big\}\]
\end{itemize}

We can turn a Kronecker representation into a shape-lemma representation by computations with univariate polynomials within complexity $O(n M(|\cP|))$: given a Kronecker representation $(Q(y), w_1(Y), \dots, w_n(Y))$ we can obtain polynomials $v_1(Y), \dots, v_n(Y)$ providing a shape-lemma representation by inverting $\frac{\partial Q}{\partial Y}(Y)$ modulo $Q(Y)$ (note that these polynomials are relatively prime) and multiplying this inverse with $w_i(Y)$ modulo $Q(Y)$ for $i=1,\dots, n$. Similarly, we may obtain a Kronecker representation from a shape-lemma representation within the same complexity bounds.

When considering sparse systems, mixed volumes appear naturally as upper bounds for the degrees of polynomials in geometric resolutions of deformation varieties.
One of the results that we will use in the complexity bounds of our algorithms is the following (see \cite[Lemma 2.3]{JMSW2009}):

\begin{lemma}\label{lem:deg_minimal}
Let $G_1,\dots, G_n\in \K[s][x_1, \dots, x_n]$ be polynomials with supports $C_1, \dots, C_n \subset (\Z_{\ge 0})^n$ containing the set  of vertices of the standard unitary simplex $\Delta_n$ of $\R^n$, and $\deg_s(G_i)\le d$ for $i=1,\dots, n$.
If $Q\in \K[s][Y]$ is a primitive minimal polynomial of a generic linear form over the set of isolated common zeros of $G_1, \dots, G_n$ over $\overline{\K(s)}$, then
\begin{align*}
\deg_Y(Q) &\le MV_n(C_1, \dots, C_n) \quad \hbox{and} \\
\deg_s(Q) &\le d \sum_{1\le i \le n} MV_n(\Delta_n, C_1, \dots, C_{i-1}, C_{i+1} , \dots, C_n).
\end{align*}
\end{lemma}

In order to simplify our complexity bounds, we will also use the following technical lemma.

\begin{lemma}\label{lem:MVbound}
Let $C_1, \dots, C_n \subset (\Z_{\ge 0})^n$ be finite sets containing the set $\Delta_n$ of the vertices of the standard unitary simplex of $\R^n$. For $i=0, \dots , n$, let $d_i \in \Z_{\ge 0}$ and set $P_0 = \conv(\{ 0 \} \times \Delta_n; (d_0, \mathbf{0}_n)) $ and $P_i = \conv(\{ 0 \} \times C_i; (d_i, \mathbf{0}_n))\subset \R^{n+1}$ for $i\ge 1$. If $d=\max\{d_i:0\le i\le n\}$, then $MV_{n+1}(P_0, P_1, \dots, P_n) \le d\cdot MV_n(C_1, \dots, C_n)$.
\end{lemma}

\begin{proof} Let $\{ e_0,e_1,\dots, e_n\}$ be the canonical basis of $\R^{n+1}$.
We have that $P_0 \subset Q_0:=\conv(0, e_1, \dots, e_n, de_0)$ and, for $i=1, \dots, n$, $P_i \subset Q_i:=\conv(\{0\} \times C_i,de_0)$ and so, due to the monotonicity of the mixed volume, the following inequality holds:
$$MV_{n+1}(P_0, P_1, \dots, P_n) \le MV_{n+1}(Q_0,Q_1, \dots, Q_n).$$
The mixed volume in the right hand side of the above inequality counts the number of common zeros of a generic sparse system of polynomials in $n+1$ variables with supports $\{0, e_1, \dots, e_n, de_0\}$ and $(\{ 0\} \times C_i) \cup \{ d e_0\}$ for $i=1, \dots, n$; that is, a system of the form
\begin{equation*}
\left\{
\begin{array}{cc}
   x_0^d+ \ell(x)& = 0 \\
   x_0^d + g_1(x)  & = 0 \\
   \vdots \\
    x_0^d + g_n(x)  & = 0
\end{array}
\right.
\end{equation*}
where $\ell(x)$ is a generic polynomial of degree $1$ and $g_1, \dots, g_n$ are generic polynomials  with supports $C_1, \dots, C_n$ in the variables $x=(x_1,\dots, x_n)$. This polynomial system is equivalent to
\begin{equation}\label{eq:systemMV}
\left\{
\begin{array}{cc}
   x_0^d+ \ell(x) & = 0 \\
   g_1(x) - \ell(x)& = 0 \\
   \vdots \\
    g_n(x)-  \ell(x)  & = 0
\end{array}
\right.
\end{equation}
By assumption, for $i=1, \dots, n$, $C_i$ contains the support of $\ell(x)$; then, $g_i(x) - \ell(x)$ is a generic polynomial with support $C_i$. Noticing that the variable $x_0$ only appears in the first equation, we conclude that the number of common solutions to the system \eqref{eq:systemMV} equals $d \cdot MV_n(C_1, \dots, C_n)$.
\end{proof}

\subsection{Algorithmic deformation to classify points}

In this section, we will present a probabilistic algorithm that, given a finite set of points $\cP \subset V$, computes, for $k=n-1,\dots, 0$, the set $\cQ_k$ consisting of all points in $\cP$ lying in the equidimensional component $V_k$ of $V$ and in no equidimensional component of $V$ of a larger dimension.

For simplicity, we assume that the number of polynomials defining $V$ is $m=n$.

Otherwise, by taking $n$ generic linear combinations of $f_1,\dots, f_m$, we may obtain a polynomial system $\widetilde{\bff} = (\widetilde{f}_1,\dots, \widetilde{f}_n)$ such that $V(\widetilde{\bff})$ has the same equidimensional components as $V$ except, possibly, for the zero-dimensional one, which may contain some points not lying in $V$. Thus, we may apply the algorithm to $\widetilde{f}_1,\dots, \widetilde{f}_n$ and remove the additional points where the system $\bff$ does not vanish. We point out that the support of the polynomials $\widetilde{f}_1,\dots, \widetilde{f}_n$ is the union of all the supports of $f_1,\dots, f_m$ and so, mixed volumes associated to this set will appear in the complexities of our algorithms.

The mixed volume associated to $n$ finite sets in $\Z^n$ can be computed as the sum of the $n$-dimensional volumes of the
convex hulls of all the \emph{mixed cells} in a fine mixed subdivision. Such a subdivision can be obtained by
means of a standard lifting process (see \cite[Section 2]{HS1995} and also \cite{Mizutani2007} for a faster algorithm computing mixed cells). Some intermediate computations of our algorithms deal with
lifting functions and mixed cells of prescribed families of sets. We will assume that they have been produced by a preprocessing, whose cost is not taken into account in our complexity estimates. Furthermore, for the sake of simplicity, we will not make explicit in our complexity estimates the dependence on the size of the objects associated to the lifting process.

\bigskip

Assume the input finite set $\cP\subset V$ is given by a shape-lemma representation
\[\cP = \{ \bfp \in \A^{n+1}\mid \bfp= (v_1(\theta), \dots, v_n(\theta)) ; M(\theta) = 0\}\]
where $M(x_0) = \prod_{\bfp \in \cP} (x_0 - \ell_0(\bfp))\in \Q[x_0]$ for a separating linear form $\ell_0\in \Q[x_0]$ (i.e. $\ell_0(\bfp) \ne \ell_0(\bfp')$ for $\bfp \ne \bfp'$ in $\cP$) and $\bfv(x_0)= (v_1(x_0), \dots, v_n(x_0))$ is a vector of $n$ polynomials in $\Q[x_0]$ such that $\bfv(\ell_0(\bfp)) = \bfp$ for every $\bfp \in \cP$.

For every $\bfp\in \cP$, let $\theta_{\bfp}= \ell_0(\bfp)$. Then, $\theta_\bfp \ne \theta_{\bfp'}$ for $\bfp \ne \bfp'$ as required in Section \ref{subsec:finiteset}.
With this choice, for a fixed $k$ with $1\le k \le n-1$, we take linear affine polynomials $\ell_1,\dots, \ell_k\in \Q[x_1,\dots, x_n]$ randomly and consider the algebraic varieties defined as in Section \ref{subsec:finiteset}:
\begin{itemize}
    \item $\Theta= \{ \theta\in \C \mid M(\theta) = 0\}$
    \item $V_{\cP} = \{ (x_0,\bfx) \in \A^{n+1} \mid M(x_0) = 0, f_1(\bfx) = 0, \dots, f_m(\bfx) = 0\}$
    \item $\cL_\cP = \{(x_0,\bfx,s) \in \A^{n+2} \mid M(x_0) = 0, \ell_i(\bfx) -(1-s) \ell_i(\bfv(x_0)) = 0 \hbox{ for all } 1\le i \le k\}$
\end{itemize}
In particular, $\widehat V_{\cP} \cap \cL_\cP$ is the set of solutions of the system
\begin{equation}\label{eq:defsystem}
 M(x_0) = 0, \ f_1(\bfx) = 0, \dots, f_m(\bfx) = 0,\  \ell_i(\bfx) -(1-s) \ell_i(\bfv(x_0)) = 0 \ (1\le i \le k).
 \end{equation}

\bigskip

We will now present our main subroutine which computes a subset $\cQ$ of the given finite set $\cP$ that contains all the points in $V_k \cap \cP$ and is contained in $\bigcup_{j\ge k} V_j$.

In the sequel, \texttt{SparseSolving} refers to the main algorithm from \cite{JMSW2009} adapted to deal with polynomials with coefficients in $\Q(s)$. Given an input system of $n$ polynomials in $n$ variables in sparse representation, it computes a finite set of points containing the isolated points in $\overline{\Q(s)}^n$ of the variety defined by the system. In order to keep a  better control of the degrees in the parameter $s$, we modify slightly the last step of the algorithm so that the output is in Kronecker form. Intermediate computations are done with truncated power series in $s$ and  rational functions are reconstructed by means of Pad\'e approximation.

The algorithm \texttt{SparseSolving} chooses a linear form at random in order to work with a \emph{separating} linear form.  In our computations, we will need to work with a \emph{well-separating} linear form (\cite[Section 12.5]{bpr}), which is a stronger, but still generic  condition (see \cite[Lemma 12.44]{bpr} and its proof) that will enable us to recover a geometric resolution at $s=0$ from the output of the algorithm.

Following the notation in \cite[Section 12.5]{bpr}, for a polynomial $p\in \Q(s)[Y]$, we write $\lim\limits_{s\to 0}(p)$ for the polynomial which is obtained from $p(Y)$ by evaluating at $s=0$ the coefficients of the polynomial $s^{-o}p(Y)$, where $o$ is the minimum of the orders of the coefficients of $p(Y)$.

\bigskip\newpage

\noindent \hrulefill

\noindent Algorithm \texttt{DiscardLowerDim}

\medskip
\noindent
INPUT: polynomials $f_1, \dots, f_n\in \Q[x_1,\dots, x_n]$ in sparse representation defining a variety $V\subset\A^n$, univariate polynomials $M, v_1, \dots, v_n$ in $\Q[x_0]$ representing a  finite set of points $\cP\subset V$ in shape-lemma form, an integer $k$ with $0\le k \le n-1$.

\medskip
\noindent
OUTPUT: univariate polynomials $\overline{M}, \overline{v}_{1}, \dots, \overline{v}_{n}$ in $\Q[x_0]$, such that $\overline{M}$ is a factor of $M$, representing a subset $\cQ\subset \cP$ that satisfies $\left(V_k\setminus\left( \bigcup_{j>k} V_j\right) \right)\cap \cP \subset \cQ \subset \left(\bigcup_{j\ge k} V_j \right)\cap \cP$.

\begin{enumerate}\itemsep=0pt
    \item  Choose randomly $k$ linear affine polynomials $\ell_1, \dots, \ell_k \in \Q[x_1,\dots, x_n]$.
    \item Set 
    \begin{itemize}\itemsep=0pt
        \item $S_0(x_0,x) = M(x_0)$,
        \item for $l=1, \dots, n$, $S_l (x_0,x)= f_l(x)$,
        \item for $j=1,\dots k$, $S_{n+j} (x_0,x)= \ell_j(x) -(1-s) \ell_j(\bfv(x_0)) $.
    \end{itemize}
    \item Take $n+1$ linear combinations of the polynomials $S_0,S_1, \dots, S_{n+k}$ of the form
    $F_l(x_0,x) = S_l(x_0,x) +\sum_{1\le j \le k} a_{lj} S_{n+j}(x_0,x)$, for $l=0,\dots n$,
    with randomly chosen coefficients $(a_{lj })_{0\le l \le n, 1\le j\le k}$.

    \item Apply the algorithm \texttt{SparseSolving} to the system $\bfF=(F_0, F_1, \dots, F_n)$ to compute a geometric resolution  $(m(Y), \bfw(Y))$  in $\Q(s)[Y]$ of a finite set of points $\widetilde{\cQ}$ of the variety  $V(\bfF) \subset \A^{n+1}(\overline{\Q(s)})$ containing its isolated points.

    \item  Compute a geometric resolution of $\widetilde{\cQ}\cap V(S_0, \dots, S_{n+k})$:
    \begin{itemize}\itemsep=0pt
        \item $\delta= \max \{ \deg(f_1), \dots, \deg(f_n), \deg(M)\}$,
        \item for $l=0, \dots, n+k$, $S_l^H(T,x_0,x)= T^{\delta} S_l\left(\frac{1}{T} (x_0,x)\right)$ (here $T$ is a new variable),
        \item $\overline{m}(Y)= \gcd(m(Y), S_l^H(\frac{\partial m}{\partial Y},\bfw(Y)); 0\le l \le  n+k)$ and $\widehat{m}(Y) = m(Y)/\overline{m}(Y)$,
        \item $\overline{\mu}(Y),\widehat{\mu}(Y) \in \Q(s)[Y]$ such that $\overline{\mu}(Y) \overline{m}(Y) + \widehat{\mu}(Y) \widehat{m}(Y) =1$,
        \item for $i=0, \dots, n$, $\overline w_i(Y) = \Rem\left(\widehat{\mu}(Y) w_i(Y) , \overline{m}(Y)\right)$.
    \end{itemize}

    \item Compute $(m_0(Y), \bfw_0(Y)) = \lim\limits_{s\to 0} \left( \overline{m}(Y),\overline{\bfw}(Y)\right)$.

    \item Clean multiplicities and obtain polynomials $\overline{m}_0(Y), \vartheta_0(Y), \vartheta_1(Y), \dots, \vartheta_n(Y)$ in $\Q[ Y]$ providing a geometric resolution
    of a finite set $\widehat{\mathcal{Q}}= \{ (\vartheta_0(y),\vartheta_1(y),\dots, \vartheta_n(y)) \mid \overline{m}_0(y) = 0\}$ as follows:
    \begin{enumerate}\itemsep=0pt
        \item Compute $\gcd(m_0(Y), m_0'(Y))$ and polynomials $P_0(Y)$, $P_1(Y)\in \Q[Y]$ such that $P_0(Y) m_0(Y) + P_1(Y) m_0'(Y) = \gcd(m_0(Y), m_0'(Y))$.
        \item Compute:
        \begin{itemize}\itemsep=0pt
            \item $\overline{m}_0(Y) = m_0(Y)/\gcd(m_0(Y), m_0'(Y))$,
            \item for $i=0, \dots, n$, $\vartheta_i (Y)= \Rem\left(P_1(Y) (w_{0,i}(Y)/\gcd(m_0(Y), m_0'(Y))), \overline{m}_0(Y)\right) $.
        \end{itemize}
    \end{enumerate}

    \item Obtain polynomials $\overline{M}(x_0), \overline{v}_1(x_0), \dots, \overline{v}_n(x_0)$ providing a geometric resolution of $\cQ = \pi_x(\widehat{\cQ}) \cap \cP$ by computing:
    \begin{enumerate}\itemsep=0pt
        \item $q(Y) = \gcd( \overline{m}_0(Y), v_i(\vartheta_0(Y)) - \vartheta_i(Y); 1\le i \le n)$,
        \item $\overline{M}(x_0) = \mbox{Res}_Y(x_0-\vartheta_0(Y), q(Y))$,
        \item for $i=1, \dots, n$, $\overline{v}_{i}(x_0) = \Rem(v_i(x_0),\overline{M}(x_0))$.
    \end{enumerate}
    \end{enumerate}
\noindent \hrulefill

\bigskip
\noindent\emph{Proof of correctness.}
 Provided that the linear affine polynomials $\ell_1, \dots, \ell_k$ chosen in Step 1 are sufficiently generic, the hypothesis (H1), (H2), (H3) and (H4) from Sections \ref{sec:onepoint} and \ref{subsec:finiteset} hold.
The polynomials $S_0,S_1, \dots, S_{n+k}$ constructed in Step 2 form the defining polynomial system \eqref{eq:defsystem} of the intersection $\widehat{V}_\cP \cap \cL_\cP$ for the varieties introduced in Section \ref{subsec:finiteset}.
In Steps 3--5, we consider them as polynomials in the variables $x_0, x$ with coefficients in $\Q[s]$.

In Step 3, by taking $n+1$ generic linear combinations of the polynomials in $\bfS=(S_0,S_1, \dots, S_{n+k})$, we get a new system of polynomials $\bfF=(F_0,\dots, F_n)$ in $n+1$ variables defining the same variety as $\bfS$ in $\A^{n+1}(\overline{\Q(s)})$ except, possibly, for some additional isolated points that might not be zeros of $\bfS$.

Then, in Step 4, the algorithm \texttt{SparseSolving} computes a geometric resolution of a finite set $\widetilde{\cQ}$ containing all the isolated points of $V(\bfS) $.

In Step 5, the algorithm produces a geometric resolution in Kronecker form of the subset $\overline{\cQ}= \widetilde{\cQ}\cap V(\bfS)$. To this end, we first compute the factor $\overline{m}(Y)$ of $m(Y)$ corresponding to the values of the parameter $Y$ such that $\bfw(Y)$ is associated with a point in $V(\bfS)$ and, then, we modify accordingly the polynomials that provide the Kronecker parameterization of the points.

Note that, since the polynomials $\bfS$ define the variety  $\widehat V_{\cP} \cap \cL_\cP$ (see \eqref{eq:defsystem}) and $\cZ_\cP$ is the union of all irreducible components $\mathcal{C}$ of this variety with $I(\mathcal{C}) \cap \Q[s] = \{ 0\}$, the equidimensional components of $V(\bfS)\subset \A^{n+1}(\overline{\Q(s)})$ are in one to one correspondence with the equidimensional components of $\cZ_\cP\subset \A^{n+2}$.
In particular, the set of isolated points of $V(\bfS)$ corresponds to the equidimensional component $\cZ_{\cP,1}$ of dimension $1$.
Now, since the finite set $\overline{\cQ}$ contains the isolated points of $V(\bfS)$, it corresponds to an equidimensional subvariety $\cW$ of $\cZ_\cP$ of dimension $1$ that contains $\cZ_{\cP,1}$.

Step 6 of the algorithm produces univariate polynomials associated with the intersection $\cW \cap \{ s =0\}$ provided that the linear form used in the computations of Step 4 is well-separating for $\overline{\cQ}$.
More precisely, following \cite[Section 12.5]{bpr}, the algorithm determines the limits when $s\to 0$ of the bounded points of $\overline{\cQ}$.

By \cite[Lemma 12.37]{bpr}, the roots of $m_0(Y) = \lim\limits_{s\to 0}(\overline{m}(Y))$ are the limits $\lim\limits_{s\to 0}(\gamma)$ of the bounded roots $\gamma\in \C\langle s \rangle$ of $\overline{m}(Y)$. Moreover, every root $y\in \C$ of $m_0$ has a multiplicity equal to the number of roots $\gamma$ of $\overline{m}(Y)$ such that  $\lim\limits_{s\to 0}(\gamma)= y$. Thus, the polynomial $m_0(Y)$ may not be square-free; so, in Step 7 we clean multiplicities and obtain the square-free polynomial $\overline{m}_0(Y)$ having the same roots as $m_0$.

Now, if $\overline{\bfw} = (\overline{w}_0, \overline{w}_1, \dots, \overline{w}_n)$ and $\bfw_0=(w_{0,0}, \dots, w_{0,n})$, from \cite[Lemma 12.43]{bpr} and the genericity of the linear forms involved in the computed geometric resolutions, we conclude that,
for every bounded root $\gamma$ of $\overline{m}$, if $\lim\limits_{s\to 0}(\gamma)= y$ and the multiplicity of $y$ as a root of $m_0$ is $\mu$,
the $i$th coordinate of the corresponding limit point equals
\[\lim_{s\to 0} \left(\dfrac{\overline{w}_i(\gamma)}{\frac{\partial\overline{m}}{\partial Y} (\gamma)} \right) = \dfrac{{w}_{0,i}^{(\mu-1)}(y)}{m_0^{(\mu)}(y)}. \]
Following \cite[Proposition 10]{GLS01}, it is not difficult to see that $\gcd(m_0, m_0')$ divides $w_{0,i}$ for $i=0, \dots, n$ and the polynomials
\[q_0 = \dfrac{m_0'}{\gcd(m_0, m_0')}, \quad \overline{w}_{0,i} = \dfrac{w_{0,i}}{\gcd(m_0, m_0')}, \ i=0, \dots, n\]
satisfy:
\[\dfrac{{w}_{0,i}^{(\mu-1)}(y)}{m_0^{(\mu)}(y)} = \dfrac{\overline{w}_{0,i}(y)}{q_0(y)}, \ i=0, \dots, n,\]
for every root $y\in \C$ of $m_0$. Finally, since the polynomials $P_0$ and $P_1$ computed in Step 7(a) satisfy $P_0 \overline{m}_0+ P_1 q_0 =1 $, it follows that, for every root $y\in \C$ of $\overline{m}_0$, the equality $P_1(y) = q_0(y)^{-1}$ holds.

Therefore, the polynomials $\overline{m}_0(Y), \vartheta_0(Y), \dots, \vartheta_n(Y)$ computed in Step 7 form a shape-lemma representation of a finite set of points $\widehat{\cQ}= \{ (\vartheta_0(y),  \dots, \vartheta_n(y)) \mid \overline{m}_0(y) =0\}$ such that $\widehat{\cQ} \times \{ 0\} \subset \cZ_\cP \cap \{s =0\}$ and that contains all the points $(\theta_\bfp, \bfp)$ with $(\theta_\bfp, \bfp,0)\in \cZ_{\cP, 1}$.

Finally, in Step 8 the algorithm determines the points $(x_0,x)$ of $\widehat{\cQ}$ such that $x$ is a point in the input set $\cP$, that is, it computes $\cQ =\pi_x(\widehat{\cQ}) \cap \cP$.
The points in $\widehat{\cQ}$ satisfy
$ x_i = \vartheta_i(y)$, $i=0,\dots, n$,
for $y\in \C$ such that $\overline{m}_0(y) = 0$. For such a point we have that $x\in \cP$ if and only if, in addition, $y$ satisfies that
$ x_i = v_i(\vartheta_0(y))$, for $i=1,\dots, n.$
Hence, these points correspond to the roots $y$ of $\overline{m}_0(Y)$ such that
\[v_i(\vartheta_0(y)) = \vartheta_i(y) , \quad i=1, \dots, n,\]
namely, the roots of the polynomial
$q(Y) = \gcd( \overline{m}_0(Y), v_i(\vartheta_0(Y)) - \vartheta_i(Y); 1\le i \le n)$ computed in Step 8(a). In Step 8(b), the algorithm obtains the polynomial $\overline{M}(x_0)$ which is the factor of $M(x_0)$ whose roots correspond to the points in $\cQ$,  and, finally, in Step 8(c), it reduces modulo $\overline{M}$ the parameterizations of the input geometric resolution in order to obtain a geometric resolution of $\cQ$. From Proposition \ref{prop:defseveralpoints},
we deduce that $\pi_x(\widehat{\cQ}) \subset \bigcup_{j\ge k} V_j$ and contains all the points in $\cP\cap V_k$ that do not lie in $\bigcup_{j> k} V_j$; therefore, the set $\cQ$ satisfies the desired conditions.
$\square$

\bigskip

\noindent \emph{Complexity analysis.}
Let $\cA_1, \dots, \cA_n\subset (\Z_{\ge 0})^n$ be the supports of $f_1, \dots, f_n$ respectively, and $\delta$ an upper bound for the degrees of $M, f_1, \dots, f_n$.

For $n\in \N$, we write $\mathbf{0}_{n}$ for the null vector with $n$ coordinates.

The supports of the polynomials $S_l\in \Q(s)[x_0,x_1, \dots, x_n]$ introduced in Step 2 of the algorithm satisfy:
\begin{itemize}\itemsep=0pt
    \item $\supp(S_0) \subset \{(\alpha,\mathbf{0}_n): 0\le \alpha \le \delta\}$;
    \item for $l=1,\dots, n$, $\supp(S_l)=\{0 \} \times \cA_l$;
    \item for $j=1, \dots, k$,  $\supp(S_{n+j}) \subset (\{0\} \times \Delta_n) \cup \{(\alpha,\mathbf{0}_n): 0\le \alpha \le \delta-1\}$
\end{itemize}

Then, the polynomials $F_0, \dots, F_n$ defined in Step 3 form a sparse polynomial system in $\Q(s)[x_0, \dots, x_n]$ supported on $\cB_0, \dots, \cB_n$, where
\begin{align*}
\cB_0&=  (\{0\} \times \Delta_n) \cup \{(\alpha,\mathbf{0}_n): 0\le \alpha \le \delta\}  \subset (\Z_{\ge 0})^{n+1},\\
\cB_l &= (\{ 0\} \times (\cA_l \cup \Delta_n))  \cup \{(\alpha,\mathbf{0}_n): 0\le \alpha \le \delta-1\} \subset (\Z_{\ge 0})^{n+1}, \quad \hbox{for } l=1,\dots, n.
\end{align*}
Note that $|\cB_0| = 1+n+\delta$ and, for $l=1, \dots, n$, $| \cB_l| \le | \cA_l| +n+\delta$. We can obtain a sparse representation of $F_0, \dots, F_n$ within complexity $O(n^2(n+\delta))$.

Algorithm \texttt{SparseSolving} works in two stages: first, it computes a geometric resolution of a generic sparse system of polynomials with supports $\cB_0, \dots, \cB_n$ and then, it makes a homotopic deformation to obtain a geometric resolution of the isolated zeros of the given system $F_0, \dots, F_n$. Up to a factor depending on the computation of a suitable subdivision of the family of supports, by \cite[Proposition 5.13]{JMSW2009}, the complexity of the first stage (over $\Q$) is of order
$$O((n^3N_{\cB}\log(\mathcal{Q_\cB})+n^{1+\Omega})M(D_{\cB})
(M(D_{\cB}) + M(E_{\cB}))),$$
where $N_{\cB} = \sum_{i=0}^n|\cB_i|$, $\mathcal{Q}_{\cB} = \max\{||b||: b \in \cB_i, i=0, \dots n\}$,
\[D_\cB = MV(\cB_0, \dots, \cB_n), \hbox{ and } E_{\cB} = \sum_{i=0}^n MV(\Delta_{n+1}, \cB_0, \dots, \cB_{i-1}, \cB_{i+1}, \dots, \cB_n)\] (see  Lemma \ref{lem:deg_minimal} above), and according to \cite[Proposition 6.1 and Section 6.2]{JMSW2009}, the deformation stage can be achieved within complexity
$O( (n^2N_\cB\log(\mathcal{Q_\cB})+n^{\Omega+1})M(D_\cB)M(E_\cB))$
over the base field $\Q(s)$.
Now, if $H_\cB$ is an upper bound for the degrees of the numerators and the denominators of the coefficients of the polynomials in the geometric resolution we are looking for, the complexity over $\Q$ of the deformation stage is bounded by
$$O( (n^2N_\cB\log(\mathcal{Q_\cB})+n^{1+\Omega})M(D_\cB)M(E_\cB)M(H_\cB)).$$
Following \cite[Section 2.3]{JMSW2009}, we can take $H_\cB$ to be an upper bound for the degree in the parameter $s$ of the coefficients of a primitive minimal polynomial for a generic linear form over the set of isolated roots in $\overline{\Q(s)}^{n+1}$ of the system $F_0, \dots, F_n$.

Since $\deg_s(F_l) =1$, Lemma \ref{lem:deg_minimal} implies that the geometric resolution $(m(Y), \bfw(Y))$ consists of polynomials in $\Q(s)[Y]$ with $\deg_Y(m) ,\deg_Y(w_i) \le D_{\cB} $ and coefficients whose numerators and denominators have degrees in $s$ bounded by $E_\cB$. Therefore, $H_\cB\le E_\cB$ and the overall complexity of Step 4 is
$$O((n^3N_{\cB}\log(\mathcal{Q_\cB})+n^{1+\Omega})M(D_{\cB})
(M(D_{\cB}) + M(E_{\cB})^2)).$$

Now, in Step 5 the algorithm applies the Extended Euclidean Algorithm to $m(Y)$ and a generic linear combination $\mathcal{S}(Y)$ of the polynomials $S_l^H(\frac{\partial m}{\partial Y}(Y), \bfw(Y))$, $l=0, \dots, n+k$. These polynomials have degrees in $Y$ bounded by $\delta D_{\cB}$.

In order to obtain the coefficients of $\mathcal{S}$, we apply a fast interpolation algorithm.
We first construct an slp that evaluates $\widehat{\mathcal{S}}(T,x_0,x) = \sum_{l=0}^{n+k} c_l S_l^H(T,x_0,x)$ for randomly chosen constants $c_l$ (here, $S_l^H(T,x_0,x)= T^{\delta} S_l(\frac{1}{T}(x_0,x))$ is the homogenization of $S_l$ at degree $\delta$).
Note that $\sum_{l=n+1}^{n+k} c_l S_l(x_0,x) = \ell(x) - (1-s) \ell(\bfv(x_0))$ for a suitable linear form $\ell$ (which we can compute with $O(nk)$ operations in $\Q$) and so, it can be evaluated by means of an slp of length $O(n \delta)$; the same holds for its homogenization up to degree $\delta$. On the other hand, since each of the polynomials $f_j$ can be evaluated by an slp of length $O(n \log(\deg(f_j)) |\cA_j|)$, we have an slp of length $O(n \log(\delta) N)$ for $\sum_{l=1}^n c_l S_l^H(T,x_0,x)$. Therefore, we obtain an slp of length $O(n\delta + n\log(\delta) N )$ for $\widehat{\mathcal{S}}(T,x_0,x)$.
Now we evaluate the polynomials $\frac{\partial m}{\partial Y}(Y), w_0(Y), \dots, w_n(Y)$ (whose degrees are bounded by $D_{\cB}$) in $\delta D_{\cB}$ points within $O(n \delta M(D_{\cB}))$ arithmetic operations in $\Q(s)$ and then, we evaluate $\widehat{\mathcal{S}}(T,x_0,x)$ at the points obtained with $O((n \delta + n\log(\delta) N ) \delta D_{\cB})$ additional steps. Finally, we recover the coefficients of $\mathcal{S}(Y)$ by means of fast interpolation within  $O(M(\delta D_{\cB}))$ operations in $\Q(s)$.

The numerators and denominators of the coefficients of $\overline{m}(Y),\overline{\bfw}(Y), \dots, \overline{\bfw}_n(Y) $ are polynomials in $s$ with degrees bounded by $E_{\cB}$. We will compute them by means of Pad\'e approximation.
Then, the total number of arithmetic operations in $\Q$ to obtain the polynomial $\mathcal{S}(Y)$ is
$$O((n\delta M(D_{\cB})+n\delta D_{\cB}(\log(\delta) N+ \delta)+M(\delta D_{\cB}))M(E_{\cB})).$$

The gcd computation can be achieved within $O(M(\delta D_{\cB}) M(E_{\cB}))$ arithmetic operations in $\Q$ from the dense representation of $m(Y)$ and $\mathcal{S}(Y)$.
All the remaining computations in Step 5 are divisions with remainder and gcd calculations with polynomials of degree bounded by $D_\cB$ in the variable $Y$ and, therefore, they can be done within complexity $O(nM(D_{\cB})M(E_{\cB}))$.

Finally, we reconstruct all the coefficients of the polynomials $\overline{m}(Y),\overline{\bfw}(Y), \dots, \overline{\bfw}_n(Y) $ by Pad\'e approximation within complexity $O(n D_{\cB} M(E_{\cB}))$.

Therefore, the overall complexity of Step 5 is bounded by
\[O(n \delta^2 N M(D_{\cB}) M(E_{\cB}) ).\]

In Step 6 the algorithm first determines $\varrho=\ord(\overline{m}(Y))$ as the minimum order of a coefficient of $\overline{m}$ by looking, for each coefficient, the minimum power of $s$ effectively appearing in its numerator and in its denominator. Then,  $m_0(Y)$ and $\bfw_0(Y)$ are obtained as
\[m_0(Y) = (s^{-\varrho} \, \overline{m}(Y) )|_{s=0}\quad  \hbox{and} \quad w_{0,i} (Y) =(s^{-\varrho}\, \overline{w}_{0,i}(Y))|_{s=0}, \ i=0, \dots, n,\]
that is, by evaluating each coefficient at $s=0$. This step only involves
$O(nD_B E_B)$
comparisons in $\Q$.

Step 7(a) is performed by applying the EEA to $m_0$ and $m_0'$ in $\Q[Y]$. Since $\deg(m_0) \le D_{\cB}$, the complexity of the computation is $O(M(D_{\cB}))$. The cost of the multiplications and divisions with remainder in Step 7(b) is of order $O(n M(D_{\cB})/\log(D_{\cB}))$. So, the total number of arithmetic operations required in this step is at most $O(nM(D_\cB))$.

Finally, in Step 8, the algorithm first computes the coefficients of the polynomials $v_i(\vartheta_0(Y))- v_i(Y)$, for $i=1,\dots, n$, by means of a fast multipoint evaluation and interpolation algorithm with $O(nM(\delta D_{\cB}))$ arithmetic operations. The gcd computations involved in Step 8(a) can be done recursively within the same complexity order. The polynomial $\overline{M}(x_0)$ in Step 8(b) has degree at most $\delta$; hence, it can be computed by interpolation in $\delta$ different values $x_0$, computing the corresponding resultant for each of them. This requires $O(\delta M(D_{\cB}) +  M(\delta))$ operations. Step 8(c) involves $n$ divisions of polynomials of degrees at most $\delta$, which can be computed within $O(n M(\delta)/\log(\delta))$ operations. The total complexity of Step 8 is bounded by
$O(nM(\delta D_{\cB}))$.

Before adding up the complexities of the intermediate steps in order to get an upper bound for the overall complexity of the algorithm, we will bound the parameters appearing in the previous estimates:
\begin{itemize}\itemsep=0pt
    \item $N_{\cB} = \sum_{l=0}^n |\cB_l| \le 1+(n+1)(n+\delta) + N$,
    \item $\mathcal{Q}_{\cB} \le \delta$ (since for every $a=(a_1, \dots, a_n)\in (\Z_{\ge 0})^n$,  $||a|| \le a_1+\cdots + a_n$),
    \item $D_{\cB}= MV_{n+1}(\cB_0, \cB_1, \dots, \cB_n) \le \delta \cdot MV_n(\cA_1 \cup \Delta_n, \dots, \cA_n\cup \Delta_n)$ (see Lemma \ref{lem:MVbound}),
    \item $E_{\cB}= \sum_{i=0}^n MV_{n+1}(\Delta_{n+1}, \cB_0, \dots, \cB_{i-1}, \cB_{i+1}, \dots, \cB_n) \le
    \delta \cdot( MV_n(\cA_1 \cup \Delta_n, \dots, \cA_n\cup \Delta_n) +   \sum_{i=1}^n MV_n(\Delta_n, \cA_1 \cup \Delta_n, \dots, \cA_{i-1} \cup \Delta_n, \cA_{i+1} \cup \Delta_n, \dots, \cA_n\cup \Delta_n))$ (see Lemma \ref{lem:MVbound}).
\end{itemize}
Setting
\begin{align}
    \cD &:=  MV_n(\cA_1 \cup \Delta_n, \dots, \cA_n\cup \Delta_n), \label{eq:defD}\\
    \cE &:=\cD +\sum_{i=1}^n MV_n(\Delta_n, \cA_1 \cup \Delta_n, \dots, \cA_{i-1} \cup \Delta_n, \cA_{i+1} \cup \Delta_n, \dots, \cA_n\cup \Delta_n), \label{eq:defE}
\end{align}
we conclude that the overall complexity of Algorithm \texttt{DiscardLowerDim} is of order
\begin{center}
\hfill $O(n^4 \delta \log(\delta) N M(\delta \cD ) M(\delta \cE)^2).$
 \hfill $\square$
\end{center}

%\[O(n^4 \delta \log(\delta) N M(\delta \cD ) M(\delta \cE)^2).\]
% \hfill $\square$

\medskip

Therefore, we have proved the following:

\begin{proposition}\label{prop:DLD}
Let $f_1, \dots, f_n\in \Q[x_1,\dots, x_n]$ be polynomials supported on $\cA_1, \dots, \cA_n$ in $(\Z_{\ge 0})^n$ that define an algebraic variety $V\subset\A^n$, and let $\cP \subset V$ be a finite set given by a geometric resolution $M, v_1, \dots, v_n$ in $\Q[x_0]$.
\texttt{DiscardLowerDim} is a probabilistic algorithm that, for a fixed $k$, $0\le k \le n-1$,  computes
a geometric resolution $\overline{M}, \overline{v}_{1}, \dots, \overline{v}_{n}$ in $\Q[x_0]$ (with  $\overline{M}$ a factor of $M$) for a subset $\cQ\subset \cP$ such that $\left(V_k\setminus\big( \bigcup_{j>k} V_j\big) \right)\cap \cP \subset \cQ \subset \left(\bigcup_{j\ge k} V_j \right)\cap \cP$. The complexity of the algorithm is of order $O(n^4 \delta \log(\delta) N \, M(\delta \cD ) M(\delta \cE)^2)$, where $\delta$ is an upper bound for the degrees of the polynomials $M, f_1,\dots, f_n$; $N= \sum_{l=1}^n |\cA_l|$, and $\cD$ and $\cE$ are the parameters defined in Equations \eqref{eq:defD} and \eqref{eq:defE}, respectively.
\end{proposition}

\begin{remark}\label{rem:DLDoutput}
 If the finite set $\cP\subset V$ is contained in $\bigcup_{j\le k} V_j$, Algorithm \texttt{DiscardLowerDim} computes $\cQ = \cP \cap V_k$.
\end{remark}

 As a consequence, given a finite set of points $\cP \subset V$, we may apply Algorithm \texttt{DiscardLowerDim} recursively to compute, for $k=n-1,\dots, 0$, the set $\cQ_k$ consisting of all points in $\cP$ lying in the equidimensional component $V_k$ and in no equidimensional component of $V$ of a larger dimension.

\medskip
\noindent \hrulefill

\noindent Algorithm \texttt{PointsInDim}

\medskip
\noindent
INPUT: sparse polynomials $f_1, \dots, f_n\in \Q[x_1, \dots, x_n]$ defining a variety $V\subset\A^n$, univariate polynomials $M, v_1, \dots, v_n \in \Q[x_0]$ representing a  finite set of points $\cP\subset V$ in shape-lemma form.

\medskip
\noindent
OUTPUT: for every $k=0, \dots, n-1$, univariate polynomials $M^{(k)}, v_1^{(k)}, \dots, v_n^{(k)}$ providing a shape-lemma representation of the set $\cQ_k = \{ \bfp \in \cP \mid \bfp \in V_k \hbox{ and } \bfp\notin V_j \hbox{ for all } j>k\}$.

\begin{enumerate}\itemsep=0pt
\item Set $M^{[n-1]}:=M$ and, for $i=1,\dots, n$, $v_i^{[n-1]}:= v_i$.

\item For $k=n-1, \dots, 1$ do
\begin{enumerate}\itemsep=0pt
    \item Apply Algorithm \texttt{DiscardLowerDim} to the system $f_1,\dots, f_n$, the geometric resolution $(M^{[k]},\mathbf{v}^{[k]})$, and the integer $k$.
    \\
    The output is a geometric resolution $(M^{(k)}, \mathbf{v}^{(k)})$ in $\Q[x_0]$ of a finite set $\cQ_{k}$.

    \item $M^{[k-1]}:= M^{[k]}/M^{(k)}$.
    \item For $i=1,\dots, n$, $v_i^{[k-1]}:= \text{Rem}( v_i^{[k]},  M^{[k-1]})$.
\end{enumerate}
\item $(M^{(0)}, \mathbf{v}^{(0)}):=(M^{[0]}, \mathbf{v}^{[0]})$.
\end{enumerate}
\noindent \hrulefill
\newpage

\begin{proposition}\label{prop:PID}
Let $f_1, \dots, f_n\in \Q[x_1,\dots, x_n]$ be polynomials supported on $\cA_1, \dots, \cA_n$ in $(\Z_{\ge 0})^n$ that define an algebraic variety $V\subset\A^n$, and let $\cP \subset V$ be a finite set given by a geometric resolution $M, v_1, \dots, v_n$ in $\Q[x_0]$.
\texttt{PointsInDim} is a probabilistic algorithm that, for every $k=0, \dots, n-1$,  computes a geometric resolution of the set
$\cQ_k =\{ \bfp \in \cP \mid \bfp \in V_k \hbox{ and } \bfp\notin V_j \hbox{ for all } j>k\}$ within complexity $O(n^5 \delta \log(\delta) N M(\delta \cD ) M(\delta \cE)^2)$, where $\delta$ is an upper bound for the degrees of the polynomials $M, f_1,\dots, f_n$; $N= \sum_{l=1}^n |\cA_l|$, and $\cD$ and $\cE$ are the parameters defined in Equations \eqref{eq:defD} and \eqref{eq:defE}, respectively.
\end{proposition}

\begin{proof}
\emph{Correctness.} For $k=n-1, \dots, 0$, let $\cP_k$ and $\cQ_k$ be the finite sets that represent the geometric resolutions $(M^{[k]}, \mathbf{v}^{[k]})$ and $(M^{(k)}, \mathbf{v}^{(k)})$ respectively.
We claim that, for every $k$, \[\cP_k = \cP \setminus \Big(\bigcup_{j>k} V_j \Big) \ \hbox{and} \ \cQ_k =\{ \bfp \in \cP \mid \bfp \in V_k \hbox{ and } \bfp\notin V_j \hbox{ for all } j>k\}.\]

In Step 1, we have that $\cP_{n-1} = \cP$.

Assume that $\cP_k = \cP \setminus \Big(\bigcup_{j>k} V_j \Big)$ for a given $k$. Since $\cP_k \subset \bigcup_{j\le k} V_j$, in Step 2(a), Algorithm \texttt{DiscardLowerDim} computes a geometric resolution of
$ \cP_k \cap V_k =\{ \bfp \in \cP \mid \bfp \in V_k \hbox{ and } \bfp\notin V_j \hbox{ for all } j>k\} $
(see Remark \ref{rem:DLDoutput}) in the same variable $x_0$ as the input, that is, with respect to the same separating linear form. Then, the polynomial $M^{[k-1]}$ computed in Step 2(b) is the minimal polynomial of the linear form over the set $\cP_k \setminus \cQ_k = \cP \setminus \Big(\bigcup_{j>k-1} V_j \Big) $ and,
therefore, $(M^{[k-1]}, \mathbf{v}^{[k-1]})$ is a geometric resolution of this set.

Finally, note that $\cP_0 = \cP \setminus \Big(\bigcup_{j>0} V_j \Big)$ and so, $(M^{(0)}, \mathbf{v}^{(0)})$ as defined in Step 3 of the algorithm is a geometric resolution of $\{ p\in \cP \mid p\in V_0 \hbox{ and } p\notin V_j \hbox{ for all } j>0\} $.

\medskip

\noindent \emph{Complexity analysis.} According to Proposition \ref{prop:DLD}, for every $k=n-1, \dots, 1$, the complexity of Step 2(a) is bounded by $O(n^4 \delta \log(\delta) N M(\delta \cD ) M(\delta \cE)^2)$. This dominates the overall complexity of Step 2, since the remaining computations amount to $n+1$ polynomial divisions involving univariate polynomials of degrees at most $\delta$.
Therefore, the overall complexity of Algorithm \texttt{PoinstInDim} is $O(n^5 \delta \log(\delta) N M(\delta \cD ) M(\delta \cE)^2)$.
\end{proof}

\begin{remark}
The output of Algorithm \texttt{PointsInDim} for $k=0$ is exactly the subset of the points $\bfp \in \cP$ that are isolated points of the variety $V$.
\end{remark}

\subsection{Computation of witness point sets}

Now we are ready to present an algorithm that computes witness point sets of the equidimensional components of a variety defined by a system of sparse polynomials.

Our procedure applies Algorithm \texttt{PointsInEquidComps} from \cite{HJS2013} as a subroutine. Given sparse polynomials $f_1, \dots, f_n \in \Q[x_1, \dots, x_n]$, this algorithm takes generic linear varieties $\mathbb{L}_1, \dots, \mathbb{L}_{n-1}$, with $\dim(\mathbb{L}_k) = n-k$ for every $k$, and  computes $n$ geometric resolutions $R^{(0)}, \dots, R^{(n-1)}$ such that, for $k=0,\dots, n-1$, $R^{(k)}$ represents a finite set $\mathcal{R}_k\subset \A^n$ satisfying
\begin{equation}\label{eq:PointsInEC}
V_0 \subset \mathcal{R}_0 \subset V  \quad \hbox{ and } \quad  V_k \cap \mathbb{ L}_k \subset \mathcal{R}_k \subset \Big( \bigcup_{j\ge k} V_j\Big) \cap \mathbb{L}_k \
\hbox{for } k=1,\dots, n-1.
\end{equation}
In order to recover $V_0$ and the witness sets $V_k \cap \mathbb{ L}_k$ of the equidimensional components of positive dimension of $V$, we apply Algorithm \texttt{PointsInDim}.

\bigskip
\noindent \hrulefill

\noindent Algorithm \texttt{WitnessPoints}

\medskip
\noindent
INPUT: sparse polynomials $f_1, \dots, f_n\in \Q[x_1,\dots,x_n]$ defining a variety $V\subset\A^n$.

\medskip
\noindent
OUTPUT: for every $k=0, \dots, n-1$, univariate polynomials $M^{(k)}, v_1^{(k)}, \dots, v_n^{(k)}$ providing a shape-lemma representation of a witness point set of the equidimensional component $V_k$ of $V$.

\begin{enumerate}
    \item Apply Algorithm \texttt{PointsInEquidComps} to the system $\mathbf{f}=(f_1,\dots, f_n)$.

    The output is a family of geometric resolutions $R^{(k)}$, for $k=0, \dots, n-1$.

    \item For $k=0,\dots, n-1$:
    \begin{enumerate}
        \item Apply Algorithm \texttt{PointsInDim} to the polynomials $f_1,\dots, f_n$ and the geometric resolution $R^{(k)}$
        to obtain a family  of geometric resolutions $(M^{(k, h)}, \mathbf{v}^{(k, h)})$, $h=0,\dots, n-1$.
       \item $(M^{(k)}, \mathbf{v}^{(k)}):=(M^{(k, k)}, \mathbf{v}^{(k, k)})$.
    \end{enumerate}
\end{enumerate}
\noindent \hrulefill

\begin{theorem}\label{thm: main}
Let $f_1, \dots, f_n\in \Q[x_1,\dots, x_n]$ be polynomials supported on  $\cA_1, \dots, \cA_n\subset (\Z_{\ge 0})^n$ that define an algebraic variety $V\subset\A^n$. \texttt{WitnessPoints} is a probabilistic algorithm that computes geometric resolutions $(M^{(k)}, \mathbf{v}^{(k)})$, for $k=0, \dots, n-1$, such that $(M^{(0)}, \mathbf{v}^{(0)})$ represents the set of all isolated points of $V$ and, for every $k=1,\dots, n-1$, $(M^{(k)}, \mathbf{v}^{(k)})$ represents a set of witness points of the equidimensional component of dimension $k$ of $V$.

The complexity of the algorithm is of order $O(n^6  N\,\cD \log(\cD)  M(\cD^2) M(\cD \cE)^2)$, where
$N= \sum_{l=1}^n |\cA_l|$, $\cD =  MV_n(\cA_1 \cup \Delta_n, \dots, \cA_n\cup \Delta_n)$ and $\cE=MV_n(\cA_1 \cup \Delta_n, \dots, \cA_n\cup \Delta_n) +\sum_{i=1}^n MV_n(\Delta_n, \cA_1 \cup \Delta_n, \dots, \cA_{i-1} \cup \Delta_n, \cA_{i+1} \cup \Delta_n, \dots, \cA_n\cup \Delta_n)$.
\end{theorem}

\begin{proof}
In Step 1, the algorithm computes a family of geometric resolutions $R^{(k)}$, for $k=0, \dots, n-1$,  representing in shape-lemma form finite sets $\cR_k$ of the variety $V$ that satisfy \eqref{eq:PointsInEC}.

Then, in Step 2, for every $k$, Algorithm \texttt{PointsInDim} with input $f_1,\dots, f_n$ and $R^{(k)}$ computes $n$ geometric resolutions $(M^{(k, h)}, \mathbf{v}^{(k, h)})$, for $h=n-1, \dots, 0$, providing  shape-lemma representation of the sets $\cQ_{k, h} = \{ \bfp \in \cR_k \mid \bfp \in V_h \hbox{ and } \bfp \notin V_j \hbox{ for all } j>h\}$. Now, for a generic linear variety $\mathbb{L}_k$ of codimension $k$, we have that $\mathbb{L}_k \cap V_k\cap V_j= \emptyset$, since $\dim(V_k\cap V_j)<k$; in particular,  every point $\bfp \in V_k\cap \mathbb{L}_k $ satisfies $\bfp\notin V_j$ for all $j>k$. Then, $\cQ_{k, k} = V_k\cap \mathbb{L}_k$ and so, $(M^{(k)}, \mathbf{v}^{(k)}):=(M^{(k, k)}, \mathbf{v}^{(k, k)})$ is a geometric resolution of a witness set of $V_k$.

Following \cite[Section 4.2]{HJS2013}, as $\deg(f_j) \le MV_n(\cA_j \cup \Delta_n,\Delta_n, \dots, \Delta_n) \le \cD$ for $j=1,\dots, n$, we can estimate the complexity of Step 1 in terms of the parameters $n,\cD$, and $\cE$ as $O(n^4 N \log(\cD)  M(\cD) M(\cE)+ n^2 N  M(\cD^2))$.

Each of the geometric resolutions $R^{(k)}$ computed in Step 1 is given by univariate polynomials of degrees at most $\cD$, since they are obtained by deformation of the solution set of a generic sparse polynomial system with supports $\cA_1 \cup \Delta_n, \dots, \cA_n\cup \Delta_n$. Therefore, by Proposition \ref{prop:PID}, the overall complexity of Step 2 is of order $O(n^6 N \, \cD \log(\cD)  M(\cD^2) M(\cD\cE)^2)$, which dominates the complexity of the previous step.
\end{proof}

\section{Conclusion}

We presented a new probabilistic algorithm that provides a description of the equidimensional decomposition of the affine algebraic variety defined by an arbitrary sparse system.
Each of the equidimensional components of the variety is characterized by means of a witness set, that is, a finite set consisting of its intersection with a generic linear variety of complementary dimension.

The use of deformation techniques and the codification of polynomials in sparse form and with straight-line programs enabled us to obtain a complexity which is polynomial in combinatorial invariants associated to the supports of the equations.

\end{document}